\documentclass[11pt, reqno]{amsart}
\usepackage{version}
\usepackage{amssymb}
\usepackage[final]{pdfpages}
\newtheorem{theorem}{Theorem}[section]
\newtheorem{lemma}[theorem]{Lemma}

\newtheorem{prop}[theorem]{Proposition}

\theoremstyle{definition}
\newtheorem{definition}[theorem]{Definition}

\newcommand{\bea}{\begin{eqnarray*}}
\newcommand{\eea}{\end{eqnarray*}}

\numberwithin{equation}{section}

\begin{document}
\title[]{An embedding of the unit ball that does not embed into a Loewner chain}

\author[Forn\ae ss]{J. E. Forn\ae ss}
\address{J. E. Forn\ae ss: Department of Mathematics\\
NTNU.}\email{johnefo@math.ntnu.no}

\author[Wold]{E. F. Wold}
\address{E. F. Wold: Department of Mathematics\\
University of Oslo\\
Postboks 1053 Blindern, NO-0316 Oslo, Norway.}\email{erlendfw@math.uio.no}

%
%    General info
%
\date{\today}

\begin{abstract}
We construct a holomorphic embedding $\phi:\mathbb B^3\rightarrow\mathbb C^3$ such 
that $\phi(\mathbb B^3)$ is not Runge in any strictly larger domain.  As a consequence, 
$\mathcal S\neq\mathcal S^1$ for $n=3$.
\end{abstract}

\maketitle

\section{Introduction}

Recall that a Loewner chain is a family $f_t:\mathbb B^n\rightarrow\mathbb C^n$ of holomorphic injections, $f_t(0)=0, f'(0)=e^{t}\cdot\mathrm{id}, t\in [0,\infty)$, with 
$f_t(\mathbb B^n)\subseteq f_s(\mathbb B^n)$ for $t\leq s$.   We let $\mathcal S$ denote the set of all univalent maps $f:\mathbb B^n\rightarrow\mathbb C^n$
with $f(0)=0, f'(0)=\mathrm{id}$, we let $\mathcal S^1$ denote the set of all $f\in\mathcal S$ such that $f$ embeds into a Loewner chain, \emph{i.e.}, 
$f=f_0$ where $(f_t)_{t\geq 0}$ is a Loewner chain, and finally we let $\mathcal S^0$ denote the set of all $f\in\mathcal S^1$
for whom we require that the family $(e^{-t}f_t)_{t\geq 0}$ is normal.   \

In one variable, the three sets coincide, and they are all compact. On the other 
hand, in higher dimensions, the sets $\mathcal S$ and $\mathcal S^1$ are certainly not compact, as can by seen as a consequence 
of the automorphism group of $\mathbb C^n$ being huge for $n\geq 2$.  On the other hand, it is known that $\mathcal S^0$ is compact, and so 
we get the chain of inclusions
\begin{equation}\label{inclusions}
\mathcal S^0 \subsetneq\mathcal S^1\subseteq\mathcal S.
\end{equation}
However, if $f\in\mathcal S^1$, there exist $\psi\in\mathcal I(\mathbb C^n)$ (the set of entire injective maps), and 
$g\in\mathcal S^0$ such that $f=\psi\circ g$, and so we may say that $\mathcal S^1$ splits (see e.g. \cite{Braccietal}, Theorem 2.6.), 
\begin{equation}\label{splits}
\mathcal S^1=\mathcal I(\mathbb C^n)\circ\mathcal S^0.
\end{equation}
The background for this article is that it has been unknown whether it is also the case that $\mathcal S=\mathcal I(\mathbb C^n)\circ\mathcal S^0$, 
or equivalently, whether $\mathcal S=\mathcal S^1$ (this problem was mentioned and discussed in \cite{ArosioBracciWold}).  In this context, the following closely related problem was recently posed by F. Bracci: Let $f\in\mathcal S$.
Does there exist a Fatou-Bieberbach domain $\Omega\subset\mathbb C^n$ such that $f(\mathbb B^n)$ is Runge in $\Omega$?
This turns out not to be the case.  

%A key feature of Loewner theory in dimension one, is that any embedding of the unit disk in 
%the complex plane, embeds into a Loewner chain that exhausts the plane.   In higher 
%dimensions such a result cannot hold: there exist non-Runge embeddings of 
%the unit ball into $\mathbb C^n, n\geq 2$, so this is prevented by the Doquier-Grauert theorem.  
%However, for $n\geq 2$, there exist holomorphic embeddings of $\mathbb C^n$ into $\mathbb C^n$
%which are not Runge (so called Fatou-Bieberbach domains), and motivated by problems in 
%describing the structure of univalent mappings of the ball in higher dimensions, F. Bracci
%recently raised the question: is any embedding of the unit ball in $\mathbb C^n$ Runge in some Fatou-Bieberbach domain? 
%(See \cite{Fiacchi} for a discussion of this and related things.)
%The purpose of this note is to show that this is not the case.  

\begin{theorem}\label{main}
For any $\epsilon>0$ there exists a continuous injective map $\phi:\overline{\mathbb B}^3\rightarrow\mathbb C^3$
with $\phi\in\mathcal O(\mathbb B^3)$, and such that 
\begin{itemize}
\item[(i)] $\|\phi-\mathrm{id}\|_{\overline{\mathbb B}^3}<\epsilon$, and 
\item[(ii)] if $\phi(\mathbb B^3)\subset\Omega$ is a Runge pair, then $\phi(\mathbb B^3)=\Omega$.
\end{itemize}
\end{theorem}

Since the conditions in Docquier-Grauert \cite{DocquierGrauert} (Definition 20) are satisfied for the increasing family $(f_t(\mathbb B^n))_{0\leq t\leq t_0}$
for any fixed $t_0$, and for any Loewner chain, it follows from \cite{DocquierGrauert} (Satz 17--19) that each pair $(f_0(\mathbb B^n),f_t(\mathbb B^n))$ is a Runge pair, 
and 
we get our second theorem as a corollary: 
\begin{theorem}
For $n=3$ we have that $\mathcal S\neq\mathcal S^1$.
\end{theorem}

\section{Prelimiaries}

The problem mentioned above was recently studied by Gaussier and Joi\c ta \cite{GaussierJoita}.
In particular, they studied the map

\begin{equation}\label{mainmap}
\phi(z)=(z_1,z_1z_2^2  + 2z_3z_2,z_1z_2 + z_3),
\end{equation}
although we have here changed coordinates to have a fixed point at the origin.  
(This map was constructed by John Wermer \cite{Wermer1959},\cite{Wermer1960} to produce 
a non Runge embedded polydisk in $\mathbb C^3$.)  It is straight forward to check 
that the map $\phi$ is injective holomorphic on the half space 
$$
H:=\{z\in\mathbb C^3:\mathrm{Re}(z_3)<0\}.
$$

For $0<p<1/4$ we set 
\begin{equation}\label{dp}
D_p:=\{z\in\mathbb C^3: 2\mathrm{Re}(z_3) + |z_3|^2 + p(|z_1|^2+|z_2|^2)<0\}.
\end{equation}
Then $D_p$ is biholomorphic to the unit ball, $0\in bD_p$, and $D_p\subset H$. \

The result obtained by Gaussier and Joi\c ta is the following: 
For $r>0$ sufficiently small there exists $0<\alpha<r$
such that the set 
$$
S_{r,\alpha}:=\{z\in\mathbb C^3:|z_1|=r, z_2=0, z_3=\alpha\}
$$
is contained in $\phi(D_p)$.  Note however that none of the disks 
$$
D_{r,\alpha}:=\{z\subset\mathbb C^3:|z_1|<r, z_2=0, z_3=\alpha\}
$$
are contained in $\phi(D_p)$; more specifically the point $(0,0,\alpha)$ is not 
contained in $\overline{\phi(D_p)}$, since $\phi$ restricted to the $z_3$-coordinate line is the identity.  
The following is a consequence.

\begin{prop}(Gaussier-Joi\c ta)
If $\phi(D_p)\subset \Omega$ and if $\Omega$ contains an open neighbourhood of $q$, then 
$\phi(D_p)$ is not Runge in $\Omega$.
\end{prop}

Our approach to prove Theorem \ref{main} is to produce an embedding which has similar "bad"
boundary points everywhere on the boundary of the embedded ball.  
More specifically, by a "bad" boundary point we will mean the following. 
\begin{definition}
Let $\Omega\subset\mathbb C^n$ be a domain.  We will say that a point $q\in b\Omega$
is W(ermer)-degenerate, if for any $\delta>0$ there exists an embedded holomorphic
disk $D\subset B_\delta(q)$ such that $bD\subset\Omega$ and $D\not\subseteq\overline\Omega$.
\end{definition}

\section{Modification of the Wermer map}

We let $B\subset\mathbb C^3$ denote the translated unit ball $B=\{z\in\mathbb C^3:2\mathrm{Re}(z_3) + \|z\|^2<0\}$, 
and we let $B'\subset\mathbb C^3$ denote the ball which is scaled by a factor two, $B'=\{z\in\mathbb C^3:2\mathrm{Re}(z_3) + (1/2)\|z\|^2<0\}$.

\begin{prop}\label{Wermermod}
Let $\{\alpha_1,...,\alpha_n\}\subset\overline B\setminus\{0\}$ and let $\epsilon>0$.  Then there 
exists an injective continuous map $\psi:\overline B'\rightarrow\mathbb C^3$ with $\psi\in\mathcal O(\mathbb B^3)$
such that the following holds
\begin{itemize}
\item[(i)] $\|\psi-\mathrm{id}\|_{\overline B'}<\epsilon$, 
\item[(ii)] $(\psi-\mathrm{id})(z)=O(\|z-\alpha_j\|^3)$ for $j=1,...,n$, and 
\item[(iii)] $\psi(0)=0$, and $0$ is W-degenerate for $\psi(B)$.
\end{itemize}
\end{prop}

\begin{proof}
We will compose the map 
\begin{equation}
\phi(z_1,z_2,z_3)=(z_1,z_1z_2^2+2z_3z_2,z_1z_2 + z_3)
\end{equation}
with several holomorphic embeddings to achieve the claims of the theorem.  

For $N\in\mathbb N$ we set 
$$
f_N(z_3)=\frac{1}{2z_3} + e^{N(2z_3)}(1-\frac{1}{2z_3}), 
$$
and 
$$
h^\delta_N(z_3)= \frac{1}{2\delta}\frac{e^{N(2z_3)}+(2\delta-1)}{1 + (2\delta-1)e^{N(2z_3)}} + \frac{1}{2\delta}.
$$
The maps $h^\delta_N$ map the left half space to the disk of radius $1/2\delta$ centred at the point
$1/2\delta$, and $-\infty$ is mapped to the point 1.  So for a fixed $\delta$ we have that $h^\delta_N\rightarrow 1$
uniformly on compact subsets of the left half space, as $N\rightarrow\infty$.  \

Then we set 
$$
F^{\delta}_1(z)=(\delta z_1,\delta z_2,z_3), 
$$
$$
F^N_2(z)=(z_1,z_2f_N(z_3),z_3), 
$$
$$
F^{N,\delta}_3(z)=(z_1h^\delta_N(z_3),z_2h^\delta_N(z_3),z_3).
$$
Then we set 
$$
F_{N_1,N_2,\delta_1,\delta_2}= (F_1^{\delta_1})^{-1}\circ\phi\circ F_2^{N_1}\circ F_3^{N_2,\delta_2}\circ F_1^{\delta_1}.
$$
\begin{lemma}
We have that $F_{N_1,N_2,\delta_1,\delta_2}:H\rightarrow\mathbb C^3$ is injective holomorphic for all sufficiently large $N_1\in\mathbb N$.
Moreover, we have that  $F_{N_1,N_2,\delta_1,\delta}$ converges uniformly to the identity as 
$N_1\rightarrow\infty$, $\delta_1\rightarrow 0$, $\delta_2\rightarrow 0$ and $N_2\rightarrow\infty$, 
and we may arrange that $0$ is W-degenerate for $F_{N_1,N_2,\delta_1,\delta_2}(B)$.  (We note that the 
rate of convergence of each of the quantities, depends on the previous one.)
\end{lemma}
\begin{proof}
To prove that $F_{N_1,N_2,\delta_1,\delta_2}:H\rightarrow\mathbb C^3$ is injective, we need to check that $f_N$ and $h_N^\delta$ are 
both non-zero on $\tilde H=\{\mathrm{Re}(z_3)<0\}$ for sufficiently large $N\in\mathbb N$.   We leave it to the reader to check that 
$h_N^\delta$ maps $\tilde H$ to the disk of radius $1/2\delta$ centred at the point $1/2\delta$.  \

Suppose that $f_N(z_3)=0$.   Then $z_3\neq 1/2$.  Multiplying by $2z_3$ we get that 
\begin{align*}
1 + e^{N(2z_3)}(2z_3-1) = 0 & \Rightarrow 2Nz_3 + \log (1-2z_3) = 0\\
& \Rightarrow 2z_3  + \log (1-2z_3)/N = 0.
\end{align*}
The last expression converges uniformly to the function $2z_3$ on compact subsets of 
$\{\mathrm{Re}(z_3)<1/4\}$, so for a given compact set $K$ containing $z_3=0$, the 
only zero on $K$ is the point $z_3=0$.  \

Nexy we check the convergence to the identity.   
We have 
$$
G(z):=F_2^\delta(F^N_3(F_1^\delta(z)))=(\delta_1 z_1h^{\delta_2}_{N_2}(z_3),\delta_1 z_2h^{\delta_2}_{N_2}(z_3)f_N(z_3),z_3).
$$
Further 
\begin{align*}
\phi(G(z)) = (&\delta_1 z_1h^{\delta_2}_{N_2}(z_3), \delta_1^3z_1z_2^2(h^{\delta_2}_{N_2}(z_3))^3f_{N_1}(z_3)^2 + \delta_1 z_2 h^{\delta_2}_{N_2}(z_3)f_{N_1}(z_3)(2z_3),\\
& \delta_1^2z_1z_2h^{\delta_2}_{N_2}(z_3)^2f_{N_1}(z_3) + z_3),
\end{align*}
and 
\begin{align*}
F_{N_1,N_2,\delta_1,\delta_2}(z)=(& z_1h^{\delta_2}_{N_2}(z_3), \delta_1^2z_1z_2^2(h^{\delta_2}_{N_2}(z_3))^3f_{N_1}(z_3)^2 +  z_2 h^{\delta_2}_{N_2}(z_3)f_{N_1}(z_3)(2z_3), \\
& \delta_1 z_1z_2h^{\delta_2}_{N_2}(z_3)^2f_{N_1}(z_3) + z_3).
\end{align*}
We now explain how to choose all the constants to get convergence to the identity.  
Note that $f_{N_1}(z_3)2z_3$ is bounded independently of $N_1$ and that $f_{N_1}(z_3)2z_3\rightarrow 1$
uniformly on compact subsets of $B'\setminus\{0\}$ as $N_1\rightarrow\infty$.  So fix a large $N_1$.
This will cause the other terms containing $f_{N_1}(z_3)$ to grow, but this growth may now 
be eliminated by choosing $\delta_1$ small.    Next, before choosing $\delta_2$ we consider the image 
$G(B)$ near the origin after such a choice is made.  Note that $f_{N_1}(0)=(1-N_1)$ and that $h^{\delta_2}_{N_2}(0)=1/\delta_2$.  This implies that $G(B)$ has 
a defining function 
$$
2\mathrm{Re}(z_3) + |z_3|^2  + \delta_2/\delta_1 |z_1|^2 + (\delta^2_2/(N_1-1)^2\delta_1)|z_2|^2 + O(|z_3|(|z_1|^2+|z_2|^2))<0.
$$
So if we choose $\delta_2$ sufficiently small, we see that $D_p\subset G(B)$, which will cause the origin
to be a W-degenerate point for $F_{N_1,N_2,\delta_1,\delta_2}(B)$.  Choosing a small 
$\delta_2$ will cause growth in all the other terms containing $h^{\delta_2}_{N_2}$, but 
this is finally "localised" to the origin by choosing a sufficiently large $N_2$.
\end{proof}

Due to the lemma, we have now proved the proposition except for the claim (ii).   However, it is 
easy to explicitly construct an interpolation operator depending continuously on the input, that corrects the map at the points 
$\alpha_1,...,\alpha_n$.

\end{proof}
\section{Proof of Theorem \ref{main}}

Let $\{\alpha_0,\alpha_1,\alpha_2,...\}\subset b\mathbb B^3$ be a dense set of points.   For $R>1$ we will let $B(j,R)$
denote the ball in $\mathbb C^3$ containing $\mathbb B^3$ with the common boundary point $\alpha_j$.
Then, for $\delta>0, n\in\mathbb N$, we set 
$$
\Omega(R,n,\delta)=\mathbb B^3(\delta)\setminus \cup_{j=1}^n B(j,R)^c).
$$
We will construct by induction embedded holomorphic disks $D_j\in\mathbb C^3$, $\delta_j,\epsilon_j>0$, and injective continuous maps $\phi_j:\overline\Omega(1+1/j,j,\delta_j)\rightarrow\mathbb C^3$,
$\psi_j\in\mathcal O(\Omega(1+1/j,j,\delta_j))$, such that the following holds
\begin{itemize}
\item[($a_n$)] $\|\phi_j-\phi_{j-1}\|_{\overline{\mathbb B}^3}<\epsilon_j$ for $j=1,2,...,n$ ($\phi_0=\mathrm{id}$),
\item[($b_n$)] $\phi_j(\alpha_k)=\phi_{j-1}(\alpha_k)$ for $k=0,1,2,...,j-1, j=1,2,...,n$,  
\item[($c_n$)] $D_j\subset D_{(1/2)^j}(\alpha_{j-1})$ for $j=1,...,n$, 
\item[($d_n$)] $bD_j\subset \phi_n(\mathbb B^3)$ for $j=1,...,n$, and 
\item[($e_n$)] $D_j\not\subseteq\overline{\phi_n(\mathbb B^3)}$ for $j=1,...,n$.
\end{itemize}
As a preliminary choice of $\{\epsilon_j\}$ we set $\epsilon_j=\epsilon\cdot (1/2)^{j+1}$.  This is just 
to ensure (i) in the theorem, as our plan is to define 
\begin{equation}\label{composition}
\phi:=\lim_{j\rightarrow\infty}\phi_j,
\end{equation}
after we explain the inductive procedure.  Each $\epsilon_j$ will however be further decreased throughout the process.  
Note in particular, that if the sequence decreases sufficiently fast, then the map $\phi:\overline{\mathbb B}\rightarrow\mathbb C^3$
will be injective.   To start the induction, we let $\phi_1$ be the map furnished by Proposition \ref{Wermermod}, creating 
a W-degenerate point at $\alpha_0$, and such 
that ($a_1$) and ($b_1$) hold. This means that there exists $D_1$ such that ($c_1$)--($e_1$) hold.  \

Assume now that ($a_n$)--($e_n$) hold for some $n\geq 1$.  Decrease $\epsilon_j$ for $j>n$ such that 
any limit $\phi$ defined as in \eqref{composition} will satisfy
\begin{itemize}
\item[($d_\infty$)] $bD_j\subset \phi(\mathbb B^3)$ for $j=1,...,n$, and 
\item[($e_\infty$)] $D_j\not\subseteq\overline{\phi(\mathbb B^3)}$ for $j=1,...,n$.
\end{itemize}
Next we let $\tilde\phi_{n+1}:\overline{\Omega(1+1/(n+1),n+1,\delta_{n+1})}\rightarrow\overline{\Omega(1+1/n,n,\delta_n)}$
be a map furnished by Proposition \ref{Wermermod}, creating a W-degenerate point at $\alpha_n$, and such that ($a_{n+1}$) and ($b_{n+1}$) holds
for the composition $\phi_{n+1}=\phi_n\circ\tilde\phi_{n+1}$.  Note that $\Omega(1+1/(n+1),n+1,\delta_{n+1})\subset\Omega(1+1/n,n,\delta_n)$
as soon as $\delta_{n+1}<\delta_n$, and note that the existence of $\tilde\phi_{n+1}$ uses both the approximation property and the 
interpolation properties at the points $\alpha_0,...,\alpha_{n-1}$.  Finally choose a disk $D_{n+1}$ such that 
($c_{n+1}$)--($e_{n+1}$) hold.  \

This completes the induction step, and we now define $\phi$ is in \eqref{composition}.   To complete
the proof, assume that  $\phi(\mathbb B^3)\subset\Omega$, 
and suppose there exists a point $p\in b\phi(\mathbb B^3)$ and a $\delta>0$ such that $B_\delta(p)\subset\Omega$.
Then by ($c_\infty$) there exists a disk $D_j\subset B_\delta(p)$ which has the properties ($d_\infty$)--($e_\infty$), 
which implies that $\phi(\mathbb B^3)$ is not Runge in $\Omega$.
$\hfill\square$

\bibliographystyle{amsplain}

\begin{thebibliography}{10}

\bibitem{ArosioBracciWold}
Arosio, L., Bracci, F., and Wold, E. F.; 
Embedding univalent functions in filtering Loewner chains in higher dimension. 
\textit{Proc. Amer. Math. Soc.} {\bf 143} (2015), no. 4, 1627--1634.

\bibitem{Braccietal}
Bracci, F., Graham, I., Hamada, H., and  Kohr, G. 
Variation of Loewner chains, extreme and support points in the class $S^0$ in higher dimensions. 
\textit{Constr. Approx.} {\bf 43} (2016), no. 2, 231--251.

\bibitem{DocquierGrauert}
Docquier, F. and Grauert, H.;  
Levisches problem und Rungescher Satz f\"{u}r Teilgebiete Steinscher Mannigfaltigkeiten, 
\textit{Math. Ann.} {\bf140} (1960) 94--123.
 

\bibitem{GaussierJoita}
Gaussier, H. and Joi\c ta, C.; On Runge neighbourhoods of closures of domains biholomorphic to a ball. 
Geometric Function Theory in Higher Dimensions, Springer INdAM Series, 2017.  


\bibitem{Wermer1959}
Wermer, J.; 
An example concerning polynomial convexity. 
\textit{Math. Ann.} {\bf 139} (1959) 147--150 

\bibitem{Wermer1960}
Wermer, J.; 
Addendum to "An example concerning polynomial convexity''. 
\textit{Math. Ann.} {\bf 140} 1960 322--323.


\end{thebibliography}

\end{document}